\documentclass[12pt,a4paper]{amsart}

\usepackage{amsmath,amsfonts,amssymb,color,extarrows}

\usepackage{tikz}

\usepackage[hmargin=2cm,vmargin=2cm]{geometry}

\usepackage{hyperref}
\usepackage{nameref,zref-xr}                    

\newcommand{\mbP}{\mathbb P}
\newcommand{\mbZ}{\mathbb Z}
\newcommand{\mbC}{\mathbb C}
\newcommand{\cP}{\mathcal P}
\newcommand{\oM}{\overline{\mathcal M}}

\newcommand{\tu}{{\widetilde u}}

\def\cM{{\mathcal{M}}}
\def\oM{{\overline{\mathcal{M}}}}

\def\d{{\partial}}

\newcommand{\eps}{\varepsilon}

\newcommand{\str}{\mathrm{str}}

\newcommand{\cA}{\mathcal A}
\newcommand{\hcA}{\widehat{\mathcal A}}
\newcommand{\DR}{\mathrm{DR}}

\newcommand{\even}{\mathrm{even}}
\newcommand{\ct}{\mathrm{ct}}
\newcommand{\cF}{\mathcal F}

\newcommand{\Coef}{\mathrm{Coef}}

\renewcommand{\top}{\mathrm{top}}
\newcommand{\red}{\mathrm{red}}

\DeclareMathOperator{\Aut}{Aut}

\newcommand{\tOmega}{\widetilde\Omega}

\newcommand{\gl}{\mathrm{gl}}

\newcommand{\mcF}{\mathcal{F}}

\newcommand{\cZ}{\mathcal{Z}}
\newcommand{\cO}{\mathcal{O}}
\newcommand{\st}{\mathrm{st}}
\newcommand{\vir}{\mathrm{vir}}
\newcommand{\cL}{\mathcal{L}}
\newcommand{\hcF}{\widehat{\mathcal{F}}}

\newcommand{\un}{{1\!\! 1}}
\renewcommand{\P}{\mathrm{P}}
\newcommand{\DRH}{\mathrm{DRH}}
\newcommand{\ta}{\widetilde{a}}
\newcommand{\tA}{\widetilde{A}}

\renewcommand{\gg}[2]{\fill[color=white] (#2) circle(2mm) node {\color{black}$\substack{#1}$}; \draw (#2) circle (2mm)}

\newcommand{\legm}[3]{\begin{scope}[shift={(#1)}] \draw (0:0) -- (#2:7.8mm);\fill[color=white] (#2:7.8mm) circle(1.7mm) node {\color{black}$\substack{#3}$};\end{scope}}
\newcommand{\lp}[2]{\begin{scope}[shift={(#1)}] \draw (#2:0.3) circle(0.3);\end{scope}}

\newtheorem{theorem}{Theorem}[section]
\newtheorem{proposition}[theorem]{Proposition}

\newtheorem{corollary}[theorem]{Corollary}
\newtheorem{conjecture}{Conjecture}

\theoremstyle{definition}

\newtheorem{definition}[theorem]{Definition}

\newtheorem{remark}[theorem]{Remark}

\numberwithin{equation}{section}

\begin{document}

\title[A generalization of Witten's conjecture for the Pixton class]{A generalization of Witten's conjecture for the Pixton class and the noncommutative KdV hierarchy}

\author{Alexandr Buryak}
\address{A. Buryak:\newline 
Faculty of Mathematics, National Research University Higher School of Economics, \newline
6 Usacheva str., Moscow, 119048, Russian Federation;\smallskip\newline 
Center for Advanced Studies, Skolkovo Institute of Science and Technology, \newline
1 Nobel str., Moscow, 143026, Russian Federation;\smallskip\newline
Faculty of Mechanics and Mathematics, Lomonosov Moscow State University, \newline 
GSP-1, 119991 Moscow, Russian Federation}
\email{aburyak@hse.ru}

\author{Paolo Rossi}
\address{P.~Rossi:\newline Dipartimento di Matematica ``Tullio Levi-Civita'', Universit\`a degli Studi di Padova,\newline Via Trieste 63, 35121 Padova, Italy}
\email{paolo.rossi@math.unipd.it}

\begin{abstract}
In this paper, we formulate and present ample evidence towards the conjecture that the partition function (i.e. the exponential of the generating series of intersection numbers with monomials in psi classes) of the Pixton class on the moduli space of stable curves is the topological tau function of the noncommutative KdV hierarchy, which we introduced in a previous work. The specialization of this conjecture to the top degree part of Pixton's class states that the partition function of the double ramification cycle is the tau function of the dispersionless limit of this hierarchy. In fact, we prove that this conjecture follows from the Double Ramification/Dubrovin--Zhang equivalence conjecture. We also provide several independent computational checks in support of it.
\end{abstract}

\date{\today}

\maketitle

\tableofcontents

\section{Introduction}

The Witten--Kontsevich theorem \cite{Wit91,Kon92} states that the partition function $\exp\left(\eps^{-2}\mcF^\mathrm{W}\right)$ of the trivial cohomological field theory on the moduli space $\oM_{g,n}$ of stable curves of genus $g$ with $n$ marked points, 
$$
\mcF^\mathrm{W}(t_0,t_1,\ldots,\eps):=\sum_{\substack{g,n\ge 0\\2g-2+n>0}}\frac{\eps^{2g}}{n!}\sum_{d_1,\ldots,d_n\ge 0}\left(\int_{\oM_{g,n}}\prod_{i=1}^n\psi_i^{d_i}\right)\prod_{i=1}^n t_{d_i},
$$
is the topological tau function of the Korteweg--de Vries hierarchy. In particular, this means that $u=u^\mathrm{W}=\frac{\d^2 \mcF^\mathrm{W}}{(\d t_0)^2}$ satisfies the infinite system of compatible PDEs
\begin{align*}
\frac{\d u}{\d t_1}=&\d_x\left(\frac{u^2}{2}+\frac{\eps^2}{12}u_{xx}\right),\quad
\frac{\d u}{\d t_2}=\d_x\left(\frac{u^3}{6}+\frac{\eps^2}{24}(2 uu_{xx}+u_x^2)+\frac{\eps^4}{240}u_{xxxx}\right),\quad \ldots,
\end{align*}
where $x=t_0$ and whose generic member is given in terms of a simple and well-known Lax representation.\\

In this paper, we propose a generalization of this result involving Pixton's class $\sum_{j=0}^g P_g^j(A)$, a family of nonhomogeneous tautological classes on $\oM_{g,n}$ depending on an $n$-tuple of integers $A=(a_1,\ldots,a_n)\in\mbZ^n$ with $\sum a_i=0$, given explicitly in terms of a subtle combinatorial formula introduced by A. Pixton, whose top degree term $P^g_g(A)$, in cohomological degree $2g$, equals (up to a constant) the double ramification cycle $\DR_g(A)$, i.e. the cohomological representative of (a compactification of) the locus of genus $g$ curves whose marked points support a principal divisor~\cite{JPPZ17}. This family of tautological classes forms a partial cohomological field theory~$c_{g,n}$ with an infinite dimensional phase space $V=\mathrm{span}(\{e_a\}_{a \in\mbZ})$. Consider the generating series
\begin{gather*}
\cF^\P(t^*_*,\eps,\mu):=\sum_{\substack{g,n\ge 0\\2g-2+n>0}}\sum_{j=0}^g\frac{\eps^{2g}\mu^{2j}}{n!}\sum_{\substack{A=(a_1,\ldots,a_n)\in\mbZ^n\\\sum a_i=0}}\sum_{d_1,\ldots,d_n\ge 0}\left(\int_{\oM_{g,n}}2^{-j}P_g^j(A)\prod_{i=1}^n\psi_i^{d_i}\right)\prod_{i=1}^n t^{a_i}_{d_i},
\end{gather*}
and let
\begin{align*}
&(w^\P)^a:=\frac{\d^2\cF^\P}{\d t^0_0\d t^{-a}_0},\quad a\in\mbZ, \qquad w^\P:=\sum_{a\in\mbZ}(w^\P)^a e^{iay}, \qquad u^\P:=\frac{S(\eps\mu\d_x)}{S(i\eps\mu\d_x\d_y)}w^\P,
\end{align*}
where $S(z):=\frac{e^{z/2}-e^{-z/2}}{z}$. Then our main Conjecture \ref{conjecture:two} states that $u=u^\P$ satisfies the infinite system of compatible PDEs
\begin{align*}
\frac{\d u}{\d t_1}=&\d_x\left(\frac{u*u}{2}+\frac{\eps^2}{12}u_{xx}\right),\\
\frac{\d u}{\d t_2}=&\d_x\left(\frac{u*u*u}{6}+\frac{\eps^2}{24}(u*u_{xx}+u_x*u_x+u_{xx}*u)+\frac{\eps^4}{240}u_{xxxx}\right),\\
\vdots\hspace{0.08cm}&
\end{align*}
where $x=t^0_0$, which is a noncommutative analogue of the KdV hierarchy above, with respect to the noncommutative Moyal product
$$
f*g:=f\ \exp\left(\frac{i \eps\mu}{2}(\overleftarrow{\d_x} \overrightarrow{\d_y}-\overleftarrow{\d_y} \overrightarrow{\d_x})\right)\ g
$$
for functions $f,g$ on a $2$-dimensional torus with coordinates $x,y$ \cite{BR21}. Notice that, together with the string and dilaton equation, the above noncommutative KdV (ncKdV) equations determine uniquely the generating series $\cF^\P$.\\

This conjecture specializes to
\begin{gather*}
\mcF^\DR(t^*_*,\eps):=\sum_{\substack{g,n\ge 0\\2g-2+n>0}}\frac{\eps^{2g}}{n!}\sum_{\substack{A=(a_1,\ldots,a_n)\in\mbZ^n\\\sum a_i=0}}\sum_{d_1,\ldots,d_n\ge 0}\left(\int_{\oM_{g,n}}\DR_g(A)\prod_{i=1}^n\psi_i^{d_i}\right)\prod_{i=1}^n t^{a_i}_{d_i}
\end{gather*}
and
\begin{align*}
(w^\DR)^a:=\frac{\d^2\mcF^\DR}{\d t^0_0\d t^{-a}_0},\quad a\in\mbZ, \qquad w^\DR:=\sum_{a\in\mbZ}(w^\DR)^a e^{iay},\qquad u^\DR:=\frac{S(\eps\d_x)}{S(i\eps\d_x\d_y)}w^\DR
\end{align*}
with $u=u^\DR|_{\eps\mapsto\eps\mu}$ satisfying the dispersionless ncKdV hierarchy
\begin{gather}
\frac{\d u}{\d t_n}=\d_x\left(\frac{u^{*(n+1)}}{(n+1)!}\right),\quad n\ge 1,
\end{gather}
which is Conjecture \ref{conjecture:one}.\\

In this paper, we provide a proof that the two above conjectures follow from the much more general DR/DZ equivalence conjecture, which states that the double ramification hierarchy (introduced in \cite{Bur15} and further studied in \cite{BR16a}) and the Dubrovin--Zhang hierarchy (introduced in \cite{DZ01}) are equivalent up to a very specific change of coordinates in the corresponding phase space \cite{Bur15,BDGR18,BDGR16,BGR19}, together with the results of \cite{BR21}, where we proved that the double ramification hierarchy for the Pixton class is indeed the ncKdV hierarchy. In particular, this proves our conjectures at the approximation up to $\eps^2$. Moreover, we provide several independent computational checks for Conjectures \ref{conjecture:one} and \ref{conjecture:two} themselves.\\

\noindent{\bf Acknowledgements}. 
The work of A.~B. (Sections 2 and 4) was supported by the grant no.~20-11-20214 of the Russian Science Foundation.\\


\section{Double ramification cycles and the dispersionless ncKdV hierarchy}

In this section, we recall the definition of the double ramification cycles on the moduli spaces of stable curves and present Conjecture~\ref{conjecture:one} describing an integrable system controlling the intersections of monomials in psi classes with the double ramification cycles. \\ 

All the cohomology and homology groups of topological spaces will be taken with complex coefficients.\\

\subsection{Double ramification cycles}

For a pair of nonnegative integers $(g,n)$ in the stable range, i.e. satisfying $2g+2-n>0$, let $\oM_{g,n}$ be the moduli space of stable algebraic curves of genus~$g$ with~$n$ marked points labeled by the set $[n]:=\{1,\ldots,n\}$. Denote by $\psi_i\in H^2(\oM_{g,n})$ the first Chern class of the line bundle $\cL_i$ over~$\oM_{g,n}$ formed by the cotangent lines at the $i$-th marked point on stable curves. The classes $\psi_i$ are called the {\it psi classes}. Denote by~$\mathbb E$ the rank~$g$ {\it Hodge vector bundle} over~$\oM_{g,n}$ whose fibers are the spaces of holomorphic one-forms on stable curves. Let $\lambda_j:=c_j(\mathbb E)\in H^{2j}(\oM_{g,n})$. Let $\cM_{g,n}\subset\oM_{g,n}$ be the moduli space of smooth pointed curves and denote by $\cM_{g,n}^\ct\subset \oM_{g,n}$ the locus of stable curves with no non-separating nodes.\\

Consider an $n$-tuple of integers $A=(a_1,\ldots,a_n)$ such that $\sum a_i=0$, it will be called a {\it vector of double ramification data}. Suppose first that not all the numbers $a_i$ are equal to zero. Let 
$$
\cZ_g(A)\subset\cM_{g,n}
$$
be the locus parameterizing the isomorphism classes of pointed smooth curves $(C;p_1,\ldots,p_n)$ satisfying the condition $\cO_C(\sum_{i=1}^n a_ip_i)\cong\cO_C$, which is algebraic and defines $\cZ_g(A)$ canonically as a substack of $\cM_{g,n}$ of dimension $2g-3+n$. Naively, the double ramification cycle $\DR_g(A)$ is defined as the cohomology class on $\oM_{g,n}$ that is Poincar\'e dual to a compactification of $\cZ(A)$ in $\oM_{g,n}$. A rigorous definition is the following (see, e.g., \cite{JPPZ17}).\\

The positive parts of $A$ define a partition $\mu=(\mu_1,\ldots,\mu_{l(\mu)})$. The negative parts of $A$ define a second partition $\nu=(\nu_1,\ldots,\nu_{l(\nu)})$. Since the parts of $A$ sum to $0$, the partitions $\mu$ and $\nu$ must be of the same size. We now allow the case $|\mu|=|\nu|=0$. Let $n_0:=n-l(\mu)-l(\nu)$. The moduli space 
$$
\oM_{g,n_0}(\mbP^1,\mu,\nu)^\sim
$$
parameterizes stable relative maps of connected algebraic curves of genus $g$ to rubber $\mbP^1$ with ramification profiles $\mu,\nu$ over the points $0,\infty\in\mbP^1$, respectively. There is a natural map 
$$
\st\colon \oM_{g,n_0}(\mbP^1,\mu,\nu)^\sim\to\oM_{g,n}
$$
forgetting everything except the marked domain curve. The moduli space $\oM_{g,n_0}(\mbP^1,\mu,\nu)^\sim$ possesses a virtual fundamental class $\left[\oM_{g,n_0}(\mbP^1,\mu,\nu)^\sim\right]^\vir$, which is a homology class of degree $2(2g-3+n)$. The {\it double ramification cycle} 
$$
\DR_g(A)\in H^{2g}(\oM_{g,n})
$$
is defined as the Poincar\'e dual to the push-forward $\st_*\left[\oM_{g,n_0}(\mbP^1,\mu,\nu)^\sim\right]^\vir\in H_{2(2g-3+n)}(\oM_{g,n})$.\\

Let us list some properties of the double ramification cycles (see, e.g., \cite{JPPZ17}). In genus~$0$, we have 
$$
\DR_0(A)=1\in H^0(\oM_{g,n}).
$$
If all the numbers $a_i$ are equal to zero, then we have
$$
\DR_g(0,\ldots,0)=(-1)^g\lambda_g\in H^{2g}(\oM_{g,n}).
$$

\bigskip

There is a very simple explicit formula for the restriction of the double ramification cycle to the moduli space $\cM_{g,n}^\ct$. For $J\subset [n]$ and $0\leq h\leq g$ in the stable range $2h-1+|J|>0$ and $2(g-h)-1+(n-|J|)>0$, denote by $\delta^J_h \in H^2(\oM_{g,n})$ the Poincar\'e dual to the substack of~$\oM_{g,n}$ formed by stable curves with a separating node at which two stable components meet, one of genus $h$ and with marked points labeled by $|J|$, and the other of genus~$g-h$ and with marked points labeled by the complement $[n]\backslash J$. We adopt the convention $\delta^J_h:=0$ if at least one of the stability conditions $2h-1+|J|>0$ and $2(g-h)-1+(n-|J|)>0$ is not satisfied. Let $a_J:=\sum_{j\in J}a_j$. Introduce a degree $2$ cohomology class $\theta_g(A)$ on $\oM_{g,n}$ by
$$
\theta_g(A):=\sum_{j=1}^n\frac{a_j^2\psi_j}{2}-\frac{1}{4}\sum_{h=0}^g\sum_{J\subset [n]}a_J^2\delta_h^J\in H^2(\oM_{g,n}).
$$
Then we have the formula
\begin{gather}\label{eq:Hain's formula}
\left.\DR_g(A)\right|_{\cM_{g,n}^\ct}=\left.\frac{1}{g!}\theta_g(A)^g\right|_{\cM_{g,n}^\ct},
\end{gather}
which is called Hain's formula. More properties of the double ramification cycles will be presented in Section~\ref{section:full Pixton's class}.\\

\subsection{The noncommutative KdV hierarchy}

The classical construction of the KdV hierarchy as the system of Lax equations (see, e.g.,~\cite{Dic03})
$$
\frac{\d L}{\d t_n}=\frac{\eps^{2n}}{(2n+1)!!}\left[\left(L^{n+1/2}\right)_+,L\right],\quad n\ge 1,
$$
where $L:=\d_x^2+2\eps^{-2}u$, $u$ is a function of $x,t_1,t_2,\ldots$, $\eps$ is a formal parameter, and $(2n+1)!!:=(2n+1)\cdot(2n-1)\cdots 3\cdot 1$, admits generalizations, called noncommutative KdV hierarchies, where one doesn't have the pairwise commutativity of the $x$-derivatives of the dependent variable~$u$. In what follows, we will work with a specific example from the class of noncommutative KdV hierarchies.\\ 

Let $u_{k_1,k_2}$, $k_1,k_2\in\mbZ_{\ge 0}$, $\eps$, and $\mu$ be formal variables and consider the space $\hcA:=\mbC[[u_{*,*},\eps,\mu]]$, whose elements will be called {\it differential polynomials in two space variables}. Consider a gradation on $\hcA$ given by 
$$
\deg u_{k_1,k_2}:=(k_1,k_2),\qquad \deg\eps:=(-1,0),\qquad \deg\mu:=(0,-1).
$$
We will denote by $\hcA^{[(d_1,d_2)]}\subset\hcA$ the space of differential polynomials of degree $(d_1,d_2)$. The space $\hcA$ is endowed with operators $\d_x$ and $\d_y$ of degrees $(1,0)$ and $(0,1)$, respectively, defined~by
\begin{gather*}
\d_x:=\sum_{k_1,k_2\geq 0} u_{k_1+1,k_2} \frac{\d}{\d u_{k_1,k_2}},\qquad \d_y:=\sum_{k_1,k_2\geq 0} u_{k_1,k_2+1} \frac{\d}{\d u_{k_1,k_2}}.
\end{gather*}
We see that $u_{k_1,k_2}=\d_x^{k_1}\d_y^{k_2}u$. We will denote $u_{0,0}$ simply by $u$.\\

The algebra $\hcA$ is also endowed with the {\it Moyal star-product} defined by
\begin{gather}\label{eq:Moyal}
f*g:=f\ \exp\left(\frac{i \eps\mu}{2}(\overleftarrow{\d_x} \overrightarrow{\d_y}-\overleftarrow{\d_y} \overrightarrow{\d_x})\right)\ g =\sum_{n\ge 0}\sum_{k_1+k_2=n}\frac{(-1)^{k_2}(i\eps\mu)^n}{2^n k_1!k_2!}(\d_x^{k_1} \d_y^{k_2} f) (\d_x^{k_2}\d_y^{k_1} g),
\end{gather}
where $f,g \in \mbC[[u_{*,*},\eps,\mu]]$. The Moyal star-product is associative and it is graded: if $\deg f = (i_1,i_2)$ and $\deg g = (j_1,j_2)$, then $\deg{(f * g)} = (i_1+j_1,i_2+j_2)$. Note also that when $\mu=0$ the Moyal star-product becomes the usual multiplication: 
\begin{gather}\label{eq:Moyal at mu=0}
\left.(f*g)\right|_{\mu=0}=f|_{\mu=0}\cdot g|_{\mu=0}.
\end{gather}

\bigskip

Let us now consider the algebra of pseudo-differential operators of the form
\begin{gather}\label{eq:pseudo-differential operator}
A=\sum_{i\le n}a_i*\d_x^i,\quad n\in\mbZ,\quad a_i\in\mbC[[u_{*,*},\mu]][[\eps,\eps^{-1}],
\end{gather}
with the multiplication $\circ$ given by
$$
(a*\d_x^i)\circ(b*\d_x^j):=\sum_{k\ge 0}{i\choose k}(a*\d_x^k b)*\d_x^{i+j-k},\quad a,b\in \mbC[[u_{*,*},\mu]][[\eps,\eps^{-1}],\quad i,j\in\mbZ.
$$
The positive part of a pseudo-differential operator~\eqref{eq:pseudo-differential operator} is defined by $A_+:=\sum_{0\le i\le n}a_i*\d_x^i$ and, as in the classical theory of pseudo-differential operators, a pseudo-differential operator $A$ of the form $\d_x^2+\sum_{i<2}a_i*\d_x^i$ has a unique square root of the form $\d_x+\sum_{i<1}b_i*\d_x^i$, which we denote by $A^{\frac{1}{2}}$.\\

Consider the operator $L:=\d_x^2+2\eps^{-2}u$. The {\it noncommutative KdV (ncKdV) hierarchy} with respect to the Moyal star-product~\eqref{eq:Moyal} is defined by (see, e.g.,~\cite{Ham05,DM00})
\begin{gather}\label{eq:ncKdV hierarchy}
\frac{\d L}{\d t_n}=\frac{\eps^{2n}}{(2n+1)!!}\left[\left(L^{n+1/2}\right)_+,L\right],\quad n\ge 1.
\end{gather}
The ncKdV hierarchy is integrable in the sense that its flows pairwise commute. Explicitly, the first two equations of the hierarchy are
\begin{align*}
\frac{\d u}{\d t_1}=&\d_x\left(\frac{u*u}{2}+\frac{\eps^2}{12}u_{xx}\right),\\
\frac{\d u}{\d t_2}=&\d_x\left(\frac{u*u*u}{6}+\frac{\eps^2}{24}(u*u_{xx}+u_x*u_x+u_{xx}*u)+\frac{\eps^4}{240}u_{xxxx}\right).
\end{align*}

\bigskip

For any $n\ge 1$, the right-hand side of~\eqref{eq:ncKdV hierarchy} has the form $\d_x P_n$, where $P_n\in\hcA^{[(0,0)]}$. Moreover, $P_n=\sum_{g=0}^n P_{n,g}$, where $P_{n,g}$ is a linear combination of the monomials $\eps^{2g}u_{d_1}*\cdots * u_{d_{n+1-g}}$ with $d_1+\cdots+d_{n+1-g}=2g$. The leading term $P_{n,0}$ is equal to $\frac{u^{*(n+1)}}{(n+1)!}$. The hierarchy 
\begin{gather}\label{eq:dncKdV}
\frac{\d u}{\d t_n}=\d_x\left(\frac{u^{*(n+1)}}{(n+1)!}\right),\quad n\ge 1,
\end{gather}
will be called the {\it dispersionless noncommutative KdV (dncKdV) hierarchy}.\\

Note that because of~\eqref{eq:Moyal at mu=0} the noncommutative KdV hierarchy becomes the classical KdV hierarchy when $\mu=0$.\\

We are now ready to present our first conjecture. Let us introduce formal variables $t^a_d$, $a\in\mbZ$, $d\ge 0$, and consider the generating function
\begin{gather*}
\mcF^\DR(t^*_*,\eps):=\sum_{\substack{g,n\ge 0\\2g-2+n>0}}\frac{\eps^{2g}}{n!}\sum_{\substack{A=(a_1,\ldots,a_n)\in\mbZ^n\\\sum a_i=0}}\sum_{d_1,\ldots,d_n\ge 0}\left(\int_{\oM_{g,n}}\DR_g(A)\prod_{i=1}^n\psi_i^{d_i}\right)\prod_{i=1}^n t^{a_i}_{d_i}\in\mbC[[t^*_*,\eps]].
\end{gather*}
Introduce a formal power series 
$$
S(z):=\frac{e^{z/2}-e^{-z/2}}{z}=1+\frac{z^2}{24}+\frac{z^4}{1920}+O(z^6)
$$
and let
\begin{align}
&(w^\DR)^a:=\frac{\d^2\mcF^\DR}{\d t^0_0\d t^{-a}_0}\in\mbC[[t^*_*,\eps]],\quad a\in\mbZ,\notag\\
&w^\DR:=\sum_{a\in\mbZ}(w^\DR)^a e^{iay}\in\mbC[[t^*_*,\eps]][[e^{iy},e^{-iy}]],\notag\\
&u^\DR:=\frac{S(\eps\d_x)}{S(i\eps\d_x\d_y)}w^\DR\in\mbC[[t^*_*,\eps]][[e^{iy},e^{-iy}]].\label{eq:uDR and wDR}
\end{align}
\begin{conjecture}\label{conjecture:one}
The function $\left.u^\DR\right|_{\eps\mapsto\eps\mu}$ satisfies the dispersionless noncommutative KdV hierarchy~\eqref{eq:dncKdV}, where we identify $t_d=t^0_d$ and $x=t_0^0$.\\
\end{conjecture}

Let us analyze the system of equations that this conjecture gives for the generating series~$\mcF^\DR$ in a bit more detail. For any $a\in\mbZ$, introduce a formal power series 
\begin{gather}\label{eq:T-function}
T(a,z):=\frac{S(z)}{S(az)}=1+\frac{1-a^2}{24}z^2+\frac{3-10a^2+7a^4}{5760}z^4+O(z^6)=\sum_{g\ge 0}Q_g(a)z^{2g}.
\end{gather}
Here, $Q_g(a)$ are polynomials in $a$. If we decompose 
\begin{gather}\label{eq:decomposition of uDR}
u^\DR=\sum_{a\in\mbZ}(u^\DR)^a e^{iay},
\end{gather}
then the transformation~\eqref{eq:uDR and wDR} simply means that
$$
(u^\DR)^a=(w^\DR)^a+\sum_{g\ge 1}\eps^{2g}Q_g(a)\d_x^{2g}(w^\DR)^a,\quad a\in\mbZ.
$$
Also, using the decomposition~\eqref{eq:decomposition of uDR} we can rewrite the equations of the dncKdV hierarchy as a system of evolutionary PDEs with one spatial variable $x$ and infinitely many times $t_d$, $d\ge 1$, for the functions $(u^\DR)^a$, $a\in\mbZ$. For example, the first equation of the dncKdV hierarchy, $\frac{\d u}{\d t_1}=\d_x\left(\frac{u*u}{2}\right)$, via Conjecture~\ref{conjecture:one}, gives the following PDEs for the functions $(u^\DR)^a$:
$$
\frac{\d (u^\DR)^a}{\d t_1}=\sum_{g\ge 0}\frac{\eps^{2g}}{2^{2g}}\sum_{\substack{a_1,a_2\in\mbZ\\a_1+a_2=a}}\sum_{\substack{k_1,k_2\ge 0\\k_1+k_2=2g}}\frac{(-1)^{k_1}}{k_1!k_2!}a_1^{k_2}a_2^{k_1}\d_x\left(\d_x^{k_1}(u^\DR)^{a_1}\d_x^{k_2}(u^\DR)^{a_2}\right),\quad a\in\mbZ.
$$
\\


\section{The Pixton class and the full ncKdV hierarchy}\label{section:full Pixton's class}

Here, we recall Pixton's very explicit construction of a nonhomogeneous cohomology class on $\oM_{g,n}$, with nontrivial terms in degree $0,2,4,\ldots,2g$. By a result of~\cite{JPPZ17}, the degree~$2g$ part of this class coincides with the double ramification cycle. We then present Conjecture~\ref{conjecture:two}, which generalizes Conjecture~\ref{conjecture:one} and says that the intersection numbers of Pixton's class with monomials in psi classes are controlled by the full noncommutative KdV hierarchy.\\

Let us first recall a standard way to construct cohomology classes on $\oM_{g,n}$ in terms of stable graphs. A {\it stable graph} is the following data:
$$
\Gamma=(V,H,L,g\colon V\to\mbZ_{\ge 0},v\colon H\to V,\iota\colon H\to H),
$$
where 
\begin{enumerate}
\item $V$ is a set of {\it vertices} with a genus function $g\colon V\to\mbZ_{\ge 0}$,
\item $H$ is a set of {\it half-edges} equipped with a vertex assignment $v\colon H\to V$ and an involution~$\iota$,
\item the set of {\it edges} $E$ is defined as the set of orbits of $\iota$ of length $2$,
\item the set of {\it legs} $L$ is defined as the set of fixed points of $\iota$ and is placed in bijective correspondence with the set $[n]$, the leg corresponding to the marking $i\in[n]$ will be denoted by $l_i$,
\item the pair $(V,E)$ defines a connected graph,
\item the stability condition $2g(v)-2+n(v)>0$ is satisfied at each vertex $v\in V$, where $n(v)$ is the valence of $\Gamma$ at $v$ including both half-edges and legs.  
\end{enumerate}
An {\it automorphism} of $\Gamma$ consists of automorphisms of the sets $V$ and $H$ that leave invariant the structures $L,g,v$, and $\iota$. Denote by $\Aut(\Gamma)$ the authomorphism group of $\Gamma$. The {\it genus} of a stable graph $\Gamma$ is defined by $g(\Gamma):=\sum_{v\in V}g(v)+h^1(\Gamma)$. Denote by $G_{g,n}$ the set of isomorphism classes of stable graphs of genus $g$ with $n$ legs.\\

For each stable graph $\Gamma\in G_{g,n}$, there is an associated moduli space
$$
\oM_\Gamma:=\prod_{v\in V}\oM_{g(v),n(v)}
$$
and a canonical map
$$
\xi_\Gamma\colon \oM_\Gamma\to\oM_{g,n}
$$
that is given by the gluing of the marked points corresponding to the two halves of each edge in $E(\Gamma)$. Each half-edge $h\in H(\Gamma)$ determines a cotangent line bundle $\cL_h\to\oM_\Gamma$. If $h\in L(\Gamma)$, then $\cL_h$ is the pull-back via $\xi_\Gamma$ of the corresponding cotangent line bundle over~$\oM_{g,n}$. Let $\psi_h:=c_1(\cL_h)\in H^2(\oM_\Gamma)$. The Pixton class will be described as a linear combination of cohomology classes of the form
$$
\xi_{\Gamma*}\left(\prod_{h\in H}\psi_h^{d(h)}\right),
$$
where $\Gamma\in G_{g,n}$ and $d\colon H(\Gamma)\to\mbZ_{\ge 0}$.\\	  

Let $A=(a_1,\ldots,a_n)$ be a vector of double ramification data. Let $\Gamma\in G_{g,n}$ and $r\ge 1$. A {\it weighting mod $r$} of $\Gamma$ is a function
$$
w\colon H(\Gamma)\to \{0,\ldots,r-1\}
$$
that satisfies the following three properties:
\begin{enumerate}
\item for any leg $l_i\in L(\Gamma)$, we have $w(l_i)=a_i\mod r$;
\item for any edge $e=\{h,h'\}\in E(\Gamma)$, we have $w(h)+w(h')=0\mod r$;
\item for any vertex $v\in V(\Gamma)$, we have $\sum_{h\in H(\Gamma),\,v(h)=v}w(h)=0\mod r$.
\end{enumerate}
Denote by $W_{\Gamma,r}$ the set of weightings mod $r$ of $\Gamma$. We have $|W_{\Gamma,r}|=r^{h^1(\Gamma)}$.\\

We denote by $P_g^{d,r}(A)\in H^{2d}(\oM_{g,n})$ the degree $2d$ component of the cohomology class 
\begin{equation}\label{eq:Pixton}
\sum_{\Gamma\in G_{g,n}}\sum_{w\in W_{\Gamma,r}}\frac{1}{|\Aut(\Gamma)|}\frac{1}{r^{h^1(\Gamma)}}\xi_{\Gamma*}\left[\prod_{i=1}^n\exp(a_i^2\psi_{l_i})\prod_{e=\{h,h'\}\in E(\Gamma)}\frac{1-\exp\left(-w(h)w(h')(\psi_h+\psi_{h'})\right)}{\psi_h+\psi_{h'}}\right]
\end{equation}
in $H^*(\oM_{g,n})$. Note that the factor $\frac{1-\exp\left(-w(h)w(h')(\psi_h+\psi_{h'})\right)}{\psi_h+\psi_{h'}}$ is well defined since the denominator formally divides the numerator. In~\cite{JPPZ17}, the authors proved that for fixed $g,A$, and $d$ the class $P_g^{d,r}$ is polynomial in~$r$ for all sufficiently large~$r$. Denote by $P_g^d(A)$ the constant term of the associated polynomial in~$r$.\\

The restriction of the class $P^j_g(A)$ to $\cM_{g,n}^\ct$ is given by
$$
\left.P^j_g(A)\right|_{\cM^\ct_{g,n}}=\left.\frac{2^j}{j!}\theta_g(A)^j\right|_{\cM^\ct_{g,n}}.
$$
In~\cite{JPPZ17}, the authors proved that 
$$
\DR_g(A)=2^{-g}P^g_g(A).
$$
In~\cite{CJ18}, the authors proved that the class $P_g^d(A)$ vanishes for $d>g$. In~\cite[page~10]{JPPZ17}, the authors remark {\it ``For $d<g$, the classes $P^d_g(A)$ do not yet have a geometric interpretation''}. Our next conjecture shows that the intersection numbers of these classes with monomials in psi classes have an elegant structure from the point of view of integrable systems.\\

Let us introduce the following generating series:
\begin{gather*}
\cF^\P(t^*_*,\eps,\mu):=\sum_{\substack{g,n\ge 0\\2g-2+n>0}}\sum_{j=0}^g\frac{\eps^{2g}\mu^{2j}}{n!}\sum_{\substack{A=(a_1,\ldots,a_n)\in\mbZ^n\\\sum a_i=0}}\sum_{d_1,\ldots,d_n\ge 0}\left(\int_{\oM_{g,n}}2^{-j}P_g^j(A)\prod_{i=1}^n\psi_i^{d_i}\right)\prod_{i=1}^n t^{a_i}_{d_i},
\end{gather*}
and let
\begin{align*}
&(w^\P)^a:=\frac{\d^2\cF^\P}{\d t^0_0\d t^{-a}_0},\quad a\in\mbZ,\\
&w^\P:=\sum_{a\in\mbZ}(w^\P)^a e^{iay},\\
&u^\P:=\frac{S(\eps\mu\d_x)}{S(i\eps\mu\d_x\d_y)}w^\P.
\end{align*}
\begin{conjecture}\label{conjecture:two}
The function $u^\P$ satisfies the full noncommutative KdV hierarchy~\eqref{eq:ncKdV hierarchy}, where we recall that we identify $t^0_d=t_d$ and $t^0_0=x$.\\
\end{conjecture}

Note that since 
$$
\left.\left(\left.\cF^\P\right|_{\substack{\eps\mapsto\eps\tau\\\mu\mapsto\tau^{-1}}}\right)\right|_{\tau=0}=\cF^\DR
$$
Conjecture~\ref{conjecture:one} immediately follows from Conjecture~2.\\

Note also that since $P^0_g(A)=1$ we have
$$
\left.\cF^\P\right|_{\mu=t^{\ne 0}_*=0}=\cF^\mathrm{W},
$$
where $\cF^{\mathrm{W}}$ is the classical generating series of intersection numbers on $\oM_{g,n}$ considered by Witten in~\cite{Wit91}:
$$
\mcF^\mathrm{W}(t_0,t_1,\ldots,\eps):=\sum_{\substack{g,n\ge 0\\2g-2+n>0}}\frac{\eps^{2g}}{n!}\sum_{d_1,\ldots,d_n\ge 0}\left(\int_{\oM_{g,n}}\prod_{i=1}^n\psi_i^{d_i}\right)\prod_{i=1}^n t_{d_i}.
$$
Clearly, we have
$$
\left.(w^\P)^a\right|_{\mu=t^{\ne 0}_*=0}=\left.(u^\P)^a\right|_{\mu=t^{\ne 0}_*=0}=
\begin{cases}
\frac{\d^2\mcF^\mathrm{W}}{(\d t_0)^2},&\text{if $a=0$},\\
0,&\text{otherwise}.
\end{cases}
$$
Thus, after the specialization $\mu=t^{\ne 0}_*=0$, Conjecture~\ref{conjecture:two} says that the function $\frac{\d^2\mcF^\mathrm{W}}{\d t_0^2}$ is a solution of the classical KdV hierarchy, which is the celebrated conjecture of Witten~\cite{Wit91}, first proved by Kontsevich~\cite{Kon92}.\\


\section{A relation with the DR/DZ equivalence conjecture}

The goal of this section is to show that Conjecture~\ref{conjecture:two} follows from the so-called DR/DZ equivalence conjecture proposed in~\cite{BDGR18} and a result of~\cite{BR21}, where the authors proved that the DR hierarchy corresponding to the partial cohomological field theory formed by the classes $\exp(\mu^2\theta_g(A))$ coincides with the noncommutative KdV hierarchy. In particular, since the DR/DZ equivalence conjecture is proved at the approximation up to genus~$1$~\cite{BDGR18,BGR19}, this proves Conjecture~\ref{conjecture:two} at the approximation up to genus~$1$.\\

\subsection{Partial cohomological field theories}

Recall the following definition, which is a generalization first considered in \cite{LRZ15} of the notion of a cohomological field theory from~\cite{KM94}.\\

\begin{definition}\label{def:partialCohFT}
A {\it partial cohomological field theory (CohFT)} is a system of linear maps 
$$
c_{g,n}\colon V^{\otimes n} \to H^\even(\oM_{g,n})
$$
for $(g,n)$ in the stable range, where $V$ is an arbitrary finite dimensional $\mbC$-vector space, called the {\it phase space}, together with a special element $e\in V$, called the {\it unit}, and a symmetric nondegenerate bilinear form $\eta\in (V^*)^{\otimes 2}$, called the {\it metric}, such that, fixing a basis $e_1,\ldots,e_{\dim V}$ in $V$, the following axioms are satisfied:
\begin{itemize}
\item[(i)] The maps $c_{g,n}$ are equivariant with respect to the $S_n$-action permuting the $n$ copies of~$V$ in $V^{\otimes n}$ and the $n$ marked points in $\oM_{g,n}$, respectively.
\item[(ii)] Let $\pi\colon\oM_{g,n+1}\to\oM_{g,n}$ be the map that forgets the last marked point. Then 
$$
\pi^* c_{g,n}( \otimes_{i=1}^n e_{\alpha_i}) = c_{g,n+1}(\otimes_{i=1}^n  e_{\alpha_i}\otimes e),\quad 1 \leq\alpha_1,\ldots,\alpha_n\leq \dim V.
$$
Moreover, $c_{0,3}(e_{\alpha}\otimes e_\beta \otimes e) =\eta(e_\alpha\otimes e_\beta) =:\eta_{\alpha\beta}$ for $1\leq \alpha,\beta\leq \dim V$.
\item[(iii)] For decompositions $I \sqcup J = [n]$, $|I|=n_1$, $|J|=n_2$, and $g_1+g_2=g$ with $2g_1-1+n_1>0$, $2g_2-1+n_2>0$, let 
\begin{gather}\label{eq:gluing map1}
\gl\colon\oM_{g_1,n_1+1}\times\oM_{g_2,n_2+1}\to \oM_{g_1+g_2,n_1+n_2}
\end{gather}
be the corresponding gluing map. Then
\begin{gather}\label{eq:gluing1 property}
\gl^* c_{g,n}( \otimes_{i=1}^n e_{\alpha_i}) = c_{g_1,n_1+1}(\otimes_{i\in I} e_{\alpha_i} \otimes e_\mu)\eta^{\mu \nu} c_{g_2,n_2+1}( \otimes_{j\in J} e_{\alpha_j}\otimes e_\nu),\quad 1 \leq\alpha_1,\ldots,\alpha_n\leq \dim V,
\end{gather}
where~$\eta^{\alpha\beta}$ are the entries of the matrix $(\eta_{\alpha\beta})^{-1}$.\\
\end{itemize}
\end{definition}

\begin{definition}
A {\it CohFT} is a partial CohFT $c_{g,n}\colon V^{\otimes n} \to H^\even(\oM_{g,n})$ such that the following extra axiom is satisfied:
\begin{itemize}
\item[(iv)] $\gl^* c_{g+1,n}(\otimes_{i=1}^n e_{\alpha_i}) = c_{g,n+2}(\otimes_{i=1}^n e_{\alpha_i}\otimes e_{\mu}\otimes e_\nu) \eta^{\mu \nu}$ for $1 \leq\alpha_1,\ldots,\alpha_n\leq \dim V$, where  $\gl\colon\oM_{g,n+2}\to \oM_{g+1,n}$ is the gluing map that increases the genus by identifying the last two marked points.\\
\end{itemize}
\end{definition}

Remark that a notion of infinite rank partial CohFT, i.e. a partial CohFT with an infinite dimensional phase space $V$, requires some care. One needs to clarify what is meant by the matrix~$(\eta^{\alpha\beta})$ and to make sense of the, a priori infinite, sum over $\mu$ and $\nu$, both appearing in Axiom~(iii). One possibility is demanding that the image of the linear map $V^{\otimes (n-1)}\to H^*(\oM_{g,n}) \otimes V^*$ induced by $c_{g,n}\colon V^{\otimes n}\to H^*(\oM_{g,n})$ is contained in $H^*(\oM_{g,n}) \otimes \eta^\sharp(V)$, where $\eta^\sharp\colon V\to V^*$ is the injective map induced by the bilinear form $\eta$. Then in Axiom~(iii), instead of using an undefined bilinear form $(\eta^{\alpha\beta})$ on $V^*$, one can use the bilinear form on $\eta^\sharp(V)$ induced by $\eta$. This solves the problem with convergence.\\

A useful special case is the following. Consider a vector space $V$ with a countable basis~$\{e_\alpha\}_{\alpha\in\mbZ}$ and suppose that for any~$(g,n)$ in the stable range and each $e_{\alpha_1},\ldots,e_{\alpha_{n-1}} \in V$ the set $\{\beta\in\mbZ\, |\, c_{g,n}(\otimes_{i=1}^{n-1} e_{\alpha_i}\otimes e_\beta)\neq 0\}$ is finite. In particular, this implies that the matrix $\eta_{\alpha\beta}$ is row- and column-finite (each row and each column have a finite number of nonzero entries), which is equivalent to $\eta^\sharp(V)\subseteq \mathrm{span}(\{e^\alpha\}_{\alpha \in \mbZ})$, where $\{e^\alpha\}_{\alpha \in \mbZ}$ is the dual ``basis''. Let us further demand that the injective map $\eta^\sharp\colon V \to \mathrm{span}(\{e^\alpha\}_{\alpha \in \mbZ})$ is surjective too, i.e. that a unique two-sided row- and column-finite matrix $(\eta^{\alpha\beta})$, inverse to $(\eta_{\alpha\beta})$, exists (it represents the inverse map $(\eta^\sharp)^{-1}\colon\mathrm{span}(\{e^\alpha\}_{\alpha \in \mbZ})\to V$). Then the equation appearing in Axiom (iii) is well defined with the double sum only having a finite number of nonzero terms. Such a partial CohFT will be called a {\it tame partial CohFT of infinite rank}.\\

\subsection{The DR/DZ equivalence conjecture}

Let us fix a positive integer $N$.\\

\subsubsection{Differential polynomials} 

Let us introduce formal variables~$u^\alpha_i$, $\alpha=1,\ldots,N$, $i=0,1,\ldots$. Following~\cite{DZ01} (see also~\cite{Ros17}), we define the ring of {\it differential polynomials} $\cA_N$ in the variables $u^1,\ldots,u^N$ as the ring of polynomials $f(u^*,u^*_1,u^*_2,\ldots)$ in the variables~$u^\alpha_i$, $i>0$, with coefficients in the ring of formal power series in the variables $u^\alpha=u^\alpha_0$:
$$
\cA_N:=\mbC[[u^*]][u^*_{\ge 1}].
$$

\bigskip

\begin{remark} This way, we define a model of the loop space of a vector space $V$ of dimension~$N$ by describing its ring of functions. In particular, it is useful to think of the variables $u^\alpha:= u^\alpha_0$ as the components $u^\alpha(x)$ of a formal loop $u\colon S^1\to V$ in a fixed basis $e_1,\ldots,e_N$ of $V$. Then the variables $u^\alpha_{1}:= u^\alpha_x, u^\alpha_{2}:= u^\alpha_{xx},\ldots$ are the components of the iterated $x$-derivatives of a formal loop.
\end{remark}

\bigskip

A gradation on $\cA_N$, which we denote by $\deg$, is introduced by $\deg u^\alpha_i:= i$. The homogeneous component of $\cA_N$ of degree $d$ is denoted by $\cA_N^{[d]}$. The operator 
$$
\partial_x := \sum_{i\geq 0} u^\alpha_{i+1}\frac{\partial}{\partial u^\alpha_i}
$$
increases the degree by $1$.\\

Differential polynomials can also be described using another set of formal variables, corresponding heuristically to the Fourier components $p^\alpha_k$, $k\in\mbZ$, of the functions $u^\alpha=u^\alpha(x)$.  We define a change of variables
\begin{gather}\label{eq:u-p change}
u^\alpha_j = \sum_{k\in\mbZ} (i k)^j p^\alpha_k e^{i k x}, 
\end{gather}
which allows us to express a differential polynomial $f(u,u_x,u_{xx},\ldots)\in\cA_N$ as a formal Fourier series in $x$. In the latter expression, the coefficient of $e^{i k x}$ is a power series in the variables~$p^\alpha_j$ with the sum of the subscripts in each monomial in $p^\alpha_j$ equal to $k$.\\

Consider the extension $\hcA_N:=\cA_N[[\eps]]$ of the space $\cA_N$ with a new variable~$\eps$ of degree $\deg\eps:=-1$. Abusing the terminology, we still call its elements {\it differential polynomials}. Let~$\hcA_N^{[k]}\subset\hcA_N$ denote the subspace of differential polynomials of degree~$k$.\\

\subsubsection{The DR hierarchy of a partial CohFT}\label{subsubsection:DR hierarchy}

Consider an arbitrary partial CohFT 
$$
c_{g,n}\colon V^{\otimes n} \to H^{\even}(\oM_{g,n}).
$$
Following~\cite{Bur15,BDGR18}, we will present the construction of the DR hierarchy and the DR/DZ equivalence conjecture. Formally, the results presented here were obtained in~\cite{Bur15,BDGR18} for a CohFT, but, as it was already remarked in~\cite[Section~9.1]{BDGR18}, the construction of the DR hierarchy works without any change for an arbitrary partial CohFT, and all the results that we discuss here are true for an arbitrary partial CohFT with the same proofs.\\ 

Let $N:=\dim V$ and let us fix a basis $e_1,\ldots,e_N\in V$. Introduce the following generating series:
\begin{gather}\label{DR polynomials}
P^\alpha_{\beta,d}:=\sum_{\substack{g\ge 0\\n\ge 1}}\frac{(-\eps^2)^g}{n!}\sum_{a_1,\ldots,a_n\in\mbZ}\left(\int_{\DR_g(-\sum_{i=1}^n a_i,0,a_1,\ldots,a_n)}\hspace{-1.2cm}\lambda_g\psi_2^d\eta^{\alpha\gamma} c_{g,n+2}(e_\gamma\otimes e_\beta\otimes_{i=1}^n e_{\alpha_i})\right)\prod_{i=1}^n p^{\alpha_i}_{a_i}e^{i(\sum_{j=1}^n a_j)x},
\end{gather}
for $\alpha,\beta=1,\ldots,N$ and $d=0,1,2,\ldots$. The expression on the right-hand side of~\eqref{DR polynomials} can be uniquely written as a differential polynomial from $\hcA^{[0]}_N$ using the change of variables~\eqref{eq:u-p change}. Concretely, it can be done in the following way. From Hain's formula~\eqref{eq:Hain's formula} it follows that the restriction $\DR_g\left(-\sum_{i=1}^n a_i,a_1,\ldots,a_n\right)\big|_{\cM_{g,n+1}^{\ct}}$ is a homogeneous polynomial in $a_1,\ldots,a_n$ of degree~$2g$ with the coefficients in $H^{2g}(\cM_{g,n+1}^{\ct})$. This property together with the fact that~$\lambda_g$ vanishes on~$\oM_{g,n}\backslash\cM_{g,n}^{\ct}$ (see, e.g.,~\cite[Section~0.4]{FP00a}) implies that the integral
\begin{gather}\label{DR integral}
\int_{\DR_g(-\sum_{i=1}^n a_i,0,a_1,\ldots,a_n)}\lambda_g\psi_2^d\eta^{\alpha\gamma} c_{g,n+2}(e_\gamma\otimes e_\beta\otimes_{i=1}^n e_{\alpha_i})
\end{gather}
is a homogeneous polynomial in $a_1,\ldots,a_n$ of degree~$2g$, which we denote by 
\begin{equation*}
Q^\alpha_{\beta,d,g;\alpha_1,\ldots,\alpha_n}(a_1,\ldots,a_n)=\sum_{\substack{b_1,\ldots,b_n\ge 0\\b_1+\ldots+b_n=2g}}Q_{\beta,d,g;\alpha_1,\ldots,\alpha_n}^{\alpha;b_1,\ldots,b_n}a_1^{b_1}\ldots a_n^{b_n}.
\end{equation*}
Then we have
$$
P^\alpha_{\beta,d}=\sum_{\substack{g\ge 0\\n\ge 1}}\frac{\eps^{2g}}{n!}\sum_{\substack{b_1,\ldots,b_n\ge 0\\b_1+\ldots+b_n=2g}}Q_{\beta,d,g;\alpha_1,\ldots,\alpha_n}^{\alpha;b_1,\ldots,b_n} u^{\alpha_1}_{b_1}\ldots u^{\alpha_n}_{b_n}.
$$
\\

The system of PDEs
\begin{gather}\label{eq:DR hierarchy}
\frac{\d u^\alpha}{\d t^\beta_d}=\d_x P^\alpha_{\beta,d},\quad 1\le\alpha,\beta\le N,\quad d\ge 0,
\end{gather}
is called the {\it double ramification (DR) hierarchy}. The flows of the hierarchy pairwise commute. Let $A^\alpha e_\alpha:=e\in V$. The flow $\frac{\d}{\d t^\un}:=A^\alpha\frac{\d}{\d t^\alpha_0}$ is given by
\begin{gather}\label{eq:tun flow}
\frac{\d u^\alpha}{\d t^\un_0}=u^\alpha_x.
\end{gather}

\bigskip

\begin{remark}
The DR hierarchy is actually Hamiltonian, and in~\cite{Bur15,BDGR18} it is introduced via a sequence of local functionals. However, since we don't need the Hamiltonian structure in this paper, we introduce directly the equations of the DR hierarchy.\\
\end{remark}

Because of~\eqref{eq:tun flow}, by a solution of the DR hierarchy we can consider an $N$-tuple of formal power series $u^\alpha(t^*_*,\eps)\in\mbC[[t^*_*,\eps]]$, $1\le\alpha\le N$, satisfying the system~\eqref{eq:DR hierarchy} after the identification of the flows $\d_x$ and $\frac{\d}{\d t^\un_0}$. The {\it string solution} $(u^\str)^\alpha(t^*_*,\eps)\in\mbC[[t^*_*,\eps]]$ of the DR heirarchy is defined as a unique solution satisfying the initial condition
$$
\left.(u^\str)^\alpha\right|_{t^\gamma_n=\delta_{n,0}A^\gamma x}=A^\alpha x.
$$ 

\bigskip

The {\it potential} of our partial CohFT is defined by
$$
\cF(t^*_*,\eps):=\sum_{\substack{g,n\ge 0\\2g-2+n>0}}\frac{\eps^{2g}}{n!}\sum_{\substack{1\le\alpha_1,\ldots,\alpha_n\le N\\d_1,\ldots,d_n\ge 0}}\left(\int_{\oM_{g,n}}c_{g,n}\left(\otimes_{i=1}^n e_{\alpha_i}\right)\prod_{i=1}^n\psi_i^{d_i}\right)\prod_{i=1}^n t^{a_i}_{d_i}\in\mbC[[t^*_*,\eps]].
$$
The exponent $\exp\left(\eps^{-2}\cF\right)$ is traditionally called the {\it partition function}.\\

Define
$$
(w^\top)^\alpha_n:=\eta^{\alpha\mu}\frac{\d^n}{(\d t^\un_0)^n}\frac{\d^2\cF}{\d t^\mu_0\d t^\un_0},\quad 1\le\alpha\le N.
$$
In~\cite[Proposition~7.2]{BDGR18}, the authors proved that there exists a unique differential polynomial $\cP\in\hcA^{[-2]}_N$ such that the power series $\cF^\red\in\mbC[[t^*_*,\eps]]$ defined by
\begin{gather*}
\cF^\red:=\cF+\left.\cP\right|_{u^\gamma_n=(w^\top)^\gamma_n}
\end{gather*}
satisfies the following vanishing property:
\begin{gather}\label{eq:property of Fred}
\Coef_{\eps^{2g}}\left.\frac{\d^n\cF^\red}{\d t^{\alpha_1}_{d_1}\ldots\d t^{\alpha_n}_{d_n}}\right|_{t^*_*=0}=0,\quad\text{if}\quad \sum_{i=1}^n d_i\le 2g-2.
\end{gather}
The power series $\cF^\red$ is called the {\it reduced potential} of our partial CohFT.\\

The differential polynomials $\tu^\alpha\in\hcA^{[0]}_N$ defined by
\begin{gather*}
\tu_\alpha:=\eta_{\un\mu}P^\mu_{\alpha,0}
\end{gather*}
are called the {\it normal coordinates} of the DR hierarchy. The differential polynomials $\tu^\alpha:=\eta^{\alpha\nu}\tu_\nu$ are also called the normal coordinates.\\

The following conjecture was presented in~\cite[Conjecture~7.5]{BDGR18}.
\begin{conjecture}\label{conjecture:DRDZ new}
We have
\begin{gather}\label{eq:DRDZ equivalence,new}
\frac{\d^2\cF^\red}{\d t^\un_0\d t^\alpha_0}=\left.\tu_\alpha\right|_{u^\gamma_n=(u^\str)^\gamma_n},\quad 1\le\alpha\le N.
\end{gather}
\end{conjecture} 

\bigskip

\begin{remark}
To be precise, Conjecture~7.5 from~\cite{BDGR18} claims that 
\begin{gather}\label{eq:DRDZ equivalence,old}
\cF^\DRH=\cF^\red,
\end{gather}
where~$\cF^\DRH$ is the potential of the DR hierarchy, see Section~4.2 in~\cite{BDGR18} for the construction. Let us explain why it is equivalent to Conjecture~\ref{conjecture:DRDZ new}. In one direction, equation~\eqref{eq:DRDZ equivalence,new} immediately follows from~\eqref{eq:DRDZ equivalence,old} and the definition of~$\cF^\DRH$. Conversely, equation~\eqref{eq:DRDZ equivalence,new} implies that $\frac{\d^2\cF^\red}{\d (t^\un_0)^2}=\frac{\d^2\cF^\DRH}{\d (t^\un_0)^2}$, which, using the string equations for $\cF^\red$ \cite[Proposition~7.2]{BDGR18} and~$\cF^\DRH$ \cite[Proposition~6.3]{BDGR18}, gives that $\cF^\red=\cF^\DRH$ (see \cite[Lemma~3.1]{Get99}).\\
\end{remark}

\subsection{The Pixton class as a partial cohomological field theory}

\begin{proposition}
The classes 
\begin{gather}\label{eq:PixtonpartialCohFT}
c_{g,n}(\otimes_{i=1}^n e_{a_i}) := \sum_{d=0}^g 2^{-d}\mu^{2d}P^d_g(a_1,\ldots,a_n)
\end{gather}
form an infinite rank tame partial cohomological field theory with the phase space $V=\mathrm{span}(\{e_a\}_{a \in \mbZ})$, the unit~$e_0$, and the metric given by in the basis $\{e_a\}_{a \in \mbZ}$ by $\eta_{a b}=\delta_{a+b,0}$.
\end{proposition}
\begin{proof}
Since $c_{g,n}(\otimes_{i=1}^n e_{a_i})=0$ unless $\sum_{i=1}^na_i=0$, the tameness property is clear.\\
 
To prove the axioms from Definition~\ref{def:partialCohFT}, the crucial observation is that formula~\eqref{eq:Pixton} is very close the formula for the action of a Givental $R$-matrix on a topological field theory (see, e.g., \cite[Section~2]{PPZ15} for an introduction to these techniques). Let $V_r:=\mathrm{span}(\{e_0,\ldots,e_{r-1}\})$ and fix a bilinear form $\eta_r(e_a,e_b):=\frac{1}{r} \delta_{a+b=0\mod r}$ on $V_r$. Starting with the topological field theory $\omega_{g,n}\colon V_r^{\otimes n}\to H^0(\oM_{g,n})$, where 
$$
\omega_{g,n}(e_{a_1}\otimes\ldots\otimes e_{a_n}):=r^{2g-1}\delta_{a_1+\ldots+a_n=0 \mod r},
$$
and applying the Givental $R$-matrix 
$$
R(z):=\exp\left(-\mathrm{diag}_{a=0}^{r-1}(a^2)2^{-1} \mu^2 z \right),
$$
we obtain the CohFT
\begin{equation*}
\begin{split}
\Omega^r_{g,n}(e_{a_1}\otimes\ldots\otimes e_{a_n})=\sum_{\Gamma\in G_{g,n}}&\sum_{w\in W_{\Gamma,r}}\frac{r^{2g-1-h^1(\Gamma)}}{|\Aut(\Gamma)|} \xi_{\Gamma*}\left[\prod_{i=1}^n \exp(a_i^2 2^{-1} \mu^2 \psi_{l_i})\right. \times\\
&\times\left.\prod_{e=\{h,h'\}\in E(\Gamma)} 2^{-1}\mu^2\frac{1-\exp\left(2^{-1}\mu^2\left(w(h)^2 \psi_h+w(h')^2 \psi_{h'}\right)\right)}{ 2^{-1}\mu^2(\psi_h+\psi_{h'})}\right],
\end{split}
\end{equation*}
whose unit is $e_0$ and where $W_{\Gamma,r}$ are the same weightings appearing in formula \eqref{eq:Pixton}. In particular, the factor $r^{2g-1-h^1(\Gamma)}$ comes from the product of the factors $r^{2g(v)-1}$ appended to each vertex $v\in V(\Gamma)$ times the factors $r$ appended to each edge (from the $\eta_r^{-1}$ in the edge contributions), since
$$
\sum_{v\in V(\Gamma)}(2g(v)-1)+|E| = |E|-|V|+2 \sum_{v\in V(\Gamma)}g(v) = (h^1(\Gamma)-1)+2(g-h^1(\Gamma)) = 2g-1-h^1(\Gamma). 
$$
Dividing the classes $\Omega^r_{g,n}(\otimes_{i=1}^n e_{a_i})$ by $r^{2g-1}$ preserves the property of being a partial CohFT. Therefore, the classes
\begin{equation}\label{eq:tOmega}
\begin{split}
\tOmega^r_{g,n}(e_{a_1}\otimes\ldots\otimes e_{a_n}):=\sum_{\Gamma\in G_{g,n}}\sum_{w\in W_{\Gamma,r}}&\frac{1}{|\Aut(\Gamma)|}\frac{1}{r^{h^1(\Gamma)}}\xi_{\Gamma*}\left[\prod_{i=1}^n \exp(a_i^2 2^{-1} \mu^2 \psi_{l_i})\right. \times\\
&\times\left.\prod_{e=\{h,h'\}\in E(\Gamma)} 2^{-1}\mu^2\frac{1-\exp\left(2^{-1}\mu^2\left(w(h)^2 \psi_h+w(h')^2 \psi_{h'}\right)\right)}{ 2^{-1}\mu^2(\psi_h+\psi_{h'})}\right]
\end{split}
\end{equation}
form a partial CohFT with the same phase space $V_r$, the metric $\widetilde{\eta}_r(e_a,e_b) = \delta_{a+b=0\mod r}$, and the unit $e_0$. Note that in this formula we have 
$$
w(h)^2=-w(h)w(h')+r w(h),\qquad w(h')^2=-w(h)w(h')+r w(h').
$$
Note also that the class $\tOmega^r_{g,n}(e_{a_1}\otimes\ldots\otimes e_{a_n})$ is zero unless $a_1+\ldots+a_n=0\mod r$.
\\

For an integer $a$, let us denote by $\ta\in\{0,\ldots,r-1\}$ the unique number such that $a=\ta\mod r$. If $r>|a|$, then, clearly, 
$$
\ta=
\begin{cases}
a,&\text{if $a\ge 0$},\\
r+a,&\text{if $a<0$}.
\end{cases}
$$
Consider an $n$-tuple $A=(a_1,\ldots,a_n)\in\mbZ^n$ satisfying $a_1+\ldots+a_n=0$, and let $\tA:=(\ta_1,\ldots,\ta_n)$. Comparing formulas~\eqref{eq:Pixton} and~\eqref{eq:tOmega}, and using Proposition 3" in~\cite{JPPZ17}, we conclude that both classes $\tOmega^r_{g,n}(\otimes_{i=1}^n e_{\ta_i})$ and $\sum_{d\ge 0}2^{-d}\mu^{2d}P_g^{d,r}(A)$ are polynomials in $r$ (for $r$ sufficiently large) having the same constant term, which is equal to the class $c_{g,n}(\otimes_{i=1}^n e_{a_i})$ (one should notice that the factors $2^{-1}\mu^2$ appended to each psi class and each edge of a stable graph in~\eqref{eq:tOmega} globally produce a factor $2^{-d}\mu^{2d}$). The proposition can now be easily derived from that.\\

To prove Axiom (iii) from Definition~\ref{def:partialCohFT} for the classes $c_{g,n}(\otimes_{i=1}^n e_{a_i})$, consider the gluing map~\eqref{eq:gluing map1} with respect to the bilinear form $\widetilde{\eta_r}$. We have
$$
\gl^*\tOmega^r_{g,n}(\otimes_{i=1}^n e_{\ta_i})=\tOmega^r_{g_1,n_1+1}(\otimes_{i\in I} e_{\ta_i}\otimes e_{\widetilde{-a_I}})\tOmega^r_{g_2,n_2+1}(\otimes_{j\in J} e_{\ta_j}\otimes e_{\widetilde{-a_J}}).
$$
Considering both sides as polynomials in $r$ (for $r$ sufficiently large) and taking the constant terms, we obtain
$$
\gl^*c_{g,n}(\otimes_{i=1}^n e_{a_i})=c_{g_1,n_1+1}(\otimes_{i\in I} e_{a_i}\otimes e_{-a_I})c_{g_2,n_2+1}(\otimes_{j\in J} e_{a_j}\otimes e_{-a_J}),
$$
as required. Proofs of Axioms~(i) and (ii) are the same, and we omit them.
\end{proof}

\bigskip

\subsection{The DR/DZ equivalence conjecture implies Conjectures~\ref{conjecture:one} and~\ref{conjecture:two}}

Consider the partial CohFT given by the Pixton class,
$$
c_{g,n}(\otimes_{i=1}^n e_{a_i}) := \sum_{d=0}^g 2^{-d}\mu^{2d}P^d_g(a_1,\ldots,a_n)
$$
and the corresponding DR hierarchy.\\

\begin{remark}
Strictly speaking, we discussed the construction of the DR hierarchy only for partial CohFTs with a finite dimensional phase space. However, it is not hard to understand that, for a partial CohFT of infinite rank, tameness is a sufficient condition for all the constructions and results to remain true. More precisely, while the definition of the Hamiltonians of the DR hierarchy works even without the tameness hypothesis for any infinite rank CohFT (at the cost of replacing the spaces of differential polynomials and local functionals with a space of formal power series in all formal variables $u^*_*$ and $\eps$), the construction of the equations of the DR hierarchy \eqref{eq:DR hierarchy} already requires dealing with the existence of the matrix $(\eta^{\alpha\beta})$ and the convergence of the infinite sum appearing in formula \eqref{DR polynomials}. From there on, through the proof of compatibility of the equations of the DR hierarchy (commutativity of Hamiltonians) to the existence of the potential of the DR hierarchy $\cF^\DRH$ featured in the DR/DZ equivalence conjecture, several constructions and results present the very same problem. It is immediate to see that the tameness hypothesis is always sufficient to ensure that $(\eta^{\alpha\beta})$ exists and that all infinite sums always have only a finite number of nonzero terms.
\end{remark}

\bigskip

\begin{proposition}\label{proposition:normal for Pixton}
The normal coordinates of the DR hierarchy are given by
\begin{gather}\label{eq:normal coordinates}
\tu^\alpha=u^\alpha+\sum_{g\ge 1}\eps^{2g}\frac{(\mu\alpha)^{2g}}{2^{2g}(2g+1)!}u^\alpha_{2g},\quad \alpha\in\mbZ.
\end{gather}
\end{proposition}
\begin{proof}
To compute the normal coordinates $\tu^\alpha$, one has to compute the integrals
\begin{gather*}
\int_{\oM_{g,n+2}}\DR_g\left(-\sum_{i=1}^n a_i,0,a_1,\ldots,a_n\right)\lambda_g 2^{-d}P^d_g(0,-\alpha,\alpha_1,\ldots,\alpha_n),
\end{gather*}
where $n\ge 1$, $a_1,\ldots,a_n,\alpha_1,\ldots,\alpha_n\in\mbZ$, $0\le d\le g$, which by degree reasons can be nonzero only if $g-1+n=d$. Therefore, only the integrals with $n=1$ and $d=g$ give a nontrivial contribution, i.e., the integrals
$$
\int_{\oM_{g,3}}\DR_g\left(-a,0,a\right)\lambda_g \DR_g(0,-\alpha,\alpha)\xlongequal{\text{\cite[Theorem~2.1]{BR21}}}\frac{(\alpha a)^{2g}}{2^{2g}(2g+1)!},
$$
which gives formula~\eqref{eq:normal coordinates}.
\end{proof}

\bigskip

\begin{proposition}\label{proposition:reduced for Pixton}
The reduced potential $\cF^{\P,\red}$ of our partial CohFT is equal to
$$
\cF^{\P,\red}=\cF^{\P}+\sum_{g\ge 1}\frac{(\eps\mu)^{2g}}{2^{2g}(2g+1)!}(w^\top)^0_{2g-2}.
$$
\end{proposition}
\begin{proof}
Equivalently, we have to check that
$$
\cF^{\P,\red}=S(\eps\mu\d_x)\cF^\P,
$$
where we identify $x=t^0_0$.\\

For $d\in\mbZ$, denote by $S_d\subset\mbC[[t^*_*,\eps]]$ the space of formal power series $F$ satisfying the condition 
$$
\Coef_{\eps^{2g}}\left.\frac{\d^n F}{\d t^{\alpha_1}_{d_1}\ldots\d t^{\alpha_n}_{d_n}}\right|_{t^*_*=0}=0,\quad\text{if}\quad\sum_{i=1}^n d_i\le 2g+d.
$$
By degree reasons, we have
$$
\cF^\P-\sum_{g\ge 1}(\eps\mu)^{2g}\bigg((-1)^g\underbrace{\int_{\oM_{g,1}}\lambda_g\psi_1^{2g-2}}_{=:b_g}\bigg)t^0_{2g-2}\in S_{-2}.
$$
Let $b_0:=1$. Using the string equation
\begin{gather*}
\frac{\d\cF^\P}{\d t^0_0}=\sum_{a\ge 0}t^\alpha_{a+1}\frac{\d\cF^\P}{\d t^\alpha_a}+\frac{1}{2}\sum_{\alpha\in\mbZ}t^\alpha_0 t^{-\alpha}_0,
\end{gather*}
we obtain
\begin{align*}
&\d_x^{2h}\cF^\P-\sum_{g\ge 0}(\eps\mu)^{2g}(-1)^g b_g t^0_{2g+2h-2}\in S_{2h-2}\quad\text{for $h\ge 1$}\quad\Rightarrow\\
\Rightarrow\quad & S(\eps\mu\d_x)\cF^\P-\sum_{g\ge 1}(\eps\mu)^{2g}t^0_{2g-2}\left(\sum_{g_1+g_2=g}\frac{(-1)^{g_1}b_{g_1}}{2^{2g_2}(2g_2+1)!}\right)\in S_{-2}.
\end{align*}
By~\cite[Theorem~2]{FP00b}, the numbers $b_g$ are given by
$$
1+\sum_{g\ge 1}b_g z^{2g}=\frac{iz}{e^{iz/2}-e^{-iz/2}}=\frac{1}{S(iz)},
$$
which implies that $S(\eps\mu\d_x)\hcF^\DR\in S_{-2}$, as required.
\end{proof}

\bigskip

\begin{theorem}\label{theorem:DRDZ implies ncKdV for Pixton}
Suppose that Conjecture~\ref{conjecture:DRDZ new} is true for the partial CohFT given by the Pixton class. Then Conjecture~\ref{conjecture:two} is true.
\end{theorem}
\begin{proof}
Conjecture~\ref{conjecture:DRDZ new} together with Proposition~\ref{proposition:normal for Pixton} implies that
$$
\frac{\d^2\cF^{\P,\red}}{\d t^0_0\d t^{-\alpha}_0}=S(\mu\eps\alpha\d_x)(u^\str)^\alpha.
$$
On the other hand, from Proposition~\ref{proposition:reduced for Pixton} it follows that
$$
\frac{\d^2\cF^{\P,\red}}{\d t^0_0\d t^{-\alpha}_0}=S(\eps\mu\d_x)(w^\P)^\alpha.
$$
Therefore,
$$
(u^\str)^\alpha=\frac{S(\eps\mu\d_x)}{S(i\eps\mu\alpha\d_x)}(w^\P)^\alpha\quad\Rightarrow\quad u^\P=\sum_{\alpha\in\mbZ}(u^\str)^\alpha e^{i\alpha y}.
$$
In~\cite[Theorem~4.1]{BR21}, the author proved that $\sum_{\alpha\in\mbZ}(u^\str)^\alpha e^{i\alpha y}$ satisfies the noncommutative KdV hierarchy. This implies that $u^\P$ satisfies the noncommutative KdV hierarchy, as required.
\end{proof}

\bigskip

\begin{corollary}
Conjectures~\ref{conjecture:one} and~\ref{conjecture:two} are true at the approximation up to $\eps^2$.
\end{corollary}
\begin{proof}
In~\cite{BGR19}, the authors proved that Conjecture~\ref{conjecture:DRDZ new} is true at the approximation up to~$\eps^2$. Together with Theorem~\ref{theorem:DRDZ implies ncKdV for Pixton}, this gives the corollary.
\end{proof}

\bigskip


\section{A prediction for the integrals $\int_{\oM_{g,2}}2^{-j}P_g^j(a,-a)\psi_1^{3g-1-j}$}

In this section, using Conjecture~\ref{conjecture:two} we will present an explicit formula for the generating series of integrals $\int_{\oM_{g,2}}2^{-j}P_g^j(a,-a)\psi_1^{3g-1-j}$. We will then check several special cases of this formula.\\

\begin{proposition}
Suppose that Conjecture~\ref{conjecture:two} is true. Then we have
\begin{gather}\label{eq:one psi for Pixton}
\sum_{1\le j\le g}\left(\int_{\oM_{g,2}}2^{-j}P_g^j(a,-a)\psi_1^{3g-1-j}\right)\mu^{2j}z^{3g-1-j}=\frac{1}{z}\left(\frac{S(a\mu z)}{S(\mu z)}e^{\frac{z^3}{24}}-1\right).
\end{gather}
\end{proposition}
\begin{proof}
By the string equation
\begin{gather}\label{eq:string for FP}
\frac{\d\cF^\P}{\d t^0_0}=\sum_{a\ge 0}t^\alpha_{a+1}\frac{\d\cF^\P}{\d t^\alpha_a}+\frac{1}{2}\sum_{\alpha\in\mbZ}t^\alpha_0 t^{-\alpha}_0,
\end{gather}
equation~\eqref{eq:one psi for Pixton} is equivalent to
\begin{align}
&\sum_{0\le j\le g}\underbrace{\left(\int_{\oM_{g,3}}2^{-j}P_g^j(a,-a,0)\psi_1^{3g-j}\right)}_{=:T_g^j(a)}\mu^{2j}z^{3g-j}=\frac{S(a\mu z)}{S(\mu z)}e^{\frac{z^3}{24}}\quad\Leftrightarrow\notag\\
\Leftrightarrow\quad&\sum_{0\le j\le g}T_g^j(a)\mu^{2j}z^{g-j}=\frac{S(a\mu)}{S(\mu)}e^{\frac{z}{24}}.\label{eq:identity for one psi for Pixton}
\end{align}

\bigskip

From~\eqref{eq:string for FP}, it also follows that
$$
\d_x^n(w^\P)^\alpha=\delta^{\alpha,0}\delta_{n,1}+\sum_{0\le j\le g}\eps^{2g}\mu^{2j}T^j_g(\alpha)t^\alpha_{3g-j+n}+O\left((t^*_*)^2\right).
$$
Therefore, $\d_x^n(u^\P)^\alpha$ has the form
\begin{gather}\label{eq:property of uP}
\d_x^n(u^\P)^\alpha=\delta^{\alpha,0}\delta_{n,1}+\sum_{0\le j\le g}\eps^{2g}\mu^{2j}R^j_g(\alpha)t^\alpha_{3g-j+n}+O\left((t^*_*)^2\right),
\end{gather}
where $R^j_g(\alpha)=\sum_{h=0}^j T^{j-h}_{g-h}(\alpha)Q_h(\alpha)$ (recall that $Q_h(\alpha)$ was defined in~\eqref{eq:T-function}) or, equivalently,
$$
\sum_{0\le j\le g}R^j_g(\alpha)\mu^{2j}z^{g-j}=\left(\sum_{0\le j\le g}T_g^j(\alpha)\mu^{2j}z^{g-j}\right)\left(\sum_{h\ge 0}Q_h(\alpha)\mu^{2h}\right).
$$
Since $\sum_{h\ge 0}Q_h(\alpha)\mu^{2h}=\frac{S(\mu)}{S(\alpha\mu)}$, we see that~\eqref{eq:identity for one psi for Pixton} is equivalent to the equation
\begin{gather}\label{eq:equation for Rjg}
R^j_g(\alpha)=\frac{\delta^{j,0}}{24^g g!}.
\end{gather}

\bigskip

Equation~\eqref{eq:equation for Rjg} is obvious for $g=0$. The property~\eqref{eq:property of uP} implies that 
$$
\frac{1}{2}\d_x(u^\P*u^\P)+\frac{\eps^2}{12}u^\P_{xxx}=\sum_{\alpha\in\mbZ}\sum_{0\le j\le g}\eps^{2g}\mu^{2j}\left(R^j_g(\alpha)+\frac{1}{12}R^j_{g-1}(\alpha)\right)t^\alpha_{3g-j}e^{i\alpha y}+O\left((t^*_*)^2\right),
$$
where we adopt the convention $R^j_g(\alpha):=0$ for $j>g$ or $g<0$. On the other hand, the dilaton equation
$$
\frac{\d\cF^\P}{\d t^0_1}=\sum_{a\ge 0}t^\alpha_a\frac{\d\cF^\P}{\d t^\alpha_a}+\eps\frac{\d\cF^\P}{\d\eps}-2\cF^\P+\frac{1}{24}
$$
implies that
$$
\frac{\d(u^\P)}{\d t^0_1}=\sum_{\alpha\in\mbZ}\sum_{0\le j\le g}\eps^{2g}\mu^{2j}(2g+1)R^j_g(\alpha)t^\alpha_{3g-j}e^{i\alpha y}+O\left((t^*_*)^2\right).
$$
Using now the first equation of the noncommutative KdV hierarchy
$$
\frac{\d u^\P}{\d t^0_1}=\frac{1}{2}(u^\P*u^\P)+\frac{\eps^2}{12}u^\P_{xxx},
$$
we obtain
$$
R^j_g(\alpha)=\frac{1}{24g}R^j_{g-1}\quad\text{for $g\ge 1$},
$$
which gives~\eqref{eq:equation for Rjg} and proves the proposition.
\end{proof}

\bigskip

Let us now check several special cases of formula~\eqref{eq:one psi for Pixton}. First of all, note that the following two specializations of~\eqref{eq:one psi for Pixton} appeared in the literature before:\\
\begin{enumerate}

\item If we put $\mu=0$ in~\eqref{eq:one psi for Pixton}, we obtain
$$
\sum_{g\ge 1}\left(\int_{\oM_{g,2}}\psi_1^{3g-1}\right)z^{3g-1}=\frac{1}{z}\left(e^{\frac{z^3}{24}}-1\right)\quad\Leftrightarrow\quad \sum_{g\ge 1}\left(\int_{\oM_{g,1}}\psi_1^{3g-2}\right)z^{3g-2}=\frac{1}{z^2}\left(e^{\frac{z^3}{24}}-1\right),
$$
which is a classical formula (see, e.g., \cite[equation~(6)]{FP00a}).\\

\item Multiplying both sides of~\eqref{eq:one psi for Pixton} by $z$, substituting $\mu\mapsto\mu z^{-1}$, and putting $z=0$, we obtain
$$
\sum_{1\le j\le g}\left(\int_{\oM_{g,2}}\DR_g(a,-a)\psi_1^{2g-1}\right)\mu^{2j}=\frac{S(a\mu)}{S(\mu)}-1,
$$
which was proved in~\cite[Theorem~1]{BSSZ15}.\\
\end{enumerate}

The two special cases discussed above involved the integrals with the classes $P^j_g(A)$ where $j=0$ or $j=g$. Consider now an example with $0<j<g$. For a fixed integer $a$ and $r$ big enough, we have
\begin{gather*}
P_2^{1,r}(a,-a)=a^2(\psi_1+\psi_2)+|a|(r-|a|)\tikz[baseline=-1mm]{\draw (0,0)--(0.8,0);\legm{(0,0)}{180}{1};\legm{0.8,0}{0}{2};\gg{1}{0,0}; \gg{1}{0.8,0};}+\frac{r^2-1}{12}\tikz[baseline=-1mm]{\lp{0,0}{90};\legm{0,0}{180}{1};\legm{0,0}{0}{2};\gg{1}{0,0};},
\end{gather*}
which gives
\begin{gather*}
P_2^1(a,-a)=a^2\left(\psi_1+\psi_2-\tikz[baseline=-1mm]{\draw (0,0)--(0.8,0);\legm{(0,0)}{180}{1};\legm{0.8,0}{0}{2};\gg{1}{0,0}; \gg{1}{0.8,0};}\right)-\frac{1}{12}\tikz[baseline=-1mm]{\lp{0,0}{90};\legm{0,0}{180}{1};\legm{0,0}{0}{2};\gg{1}{0,0};},
\end{gather*}
where we refer a reader to~\cite[Section~2.1]{BGR19} for our pictorial notation for the cohomology classes on~$\oM_{g,n}$. Therefore,
$$
\int_{\oM_{2,2}}P_2^1(a,-a)\psi_1^4=\frac{a^2-1}{288},
$$
which agrees with formula~\eqref{eq:one psi for Pixton}.
\\

\end{document}